\documentclass[11pt]{article}
\usepackage{amsmath,amssymb,amsthm,amsfonts,latexsym,xspace,enumerate,graphicx,color,eucal,upgreek,mathtools}
\usepackage[backref,colorlinks,bookmarks=true]{hyperref}

\newtheorem{theorem}{Theorem}[section]
\newtheorem*{namedtheorem}{\theoremname}
\newcommand{\theoremname}{testing}

\newtheorem{lemma}[theorem]{Lemma}

\newtheorem{claim}[theorem]{Claim}

\newtheorem{prop}[theorem]{Proposition}

\newtheorem{conjecture}{Conjecture}
\theoremstyle{definition}
\newtheorem{definition}[theorem]{Definition}

\theoremstyle{plain}
\newtheorem{Alg}{Algorithm}

\renewenvironment{proof}{\noindent{\textbf{Proof:}}} {$\blacksquare$\vskip \belowdisplayskip}

\newcommand{\ignore}[1]{}

\newcommand{\eps}{\epsilon}

\newcommand{\calC}{{\mathcal{C}}}

\newcommand{\cF}{\mathcal{F}}

\newcommand{\cP}{\mathcal{P}}

\begin{document}

\title{Nearly Complete Graphs Decomposable into 
Large Induced Matchings and their Applications}
\author{Noga Alon 
\thanks{Sackler School of Mathematics
and Blavatnik School of Computer Science,
Tel Aviv University,
Tel Aviv 69978, Israel
and
Institute for Advanced Study, Princeton, New Jersey,
08540, USA.
Email: {\tt nogaa@tau.ac.il}.
Research supported in part by an ERC Advanced
grant, by a USA-Israeli BSF grant and by NSF grant No.
DMS-0835373.}
\and Ankur Moitra 
\thanks{
Institute for Advanced Study, Princeton, New Jersey,
08540, USA.
Email: {\tt moitra@ias.edu}.
Research supported in part 
by NSF grant No.
DMS-0835373 and by an NSF Computing and Innovation Fellowship.}
\and Benny Sudakov
\thanks{Department of Mathematics, UCLA,  Los Angeles, CA
90095. Email: {\tt bsudakov@math.ucla.edu}.
Research supported in part by NSF grant DMS-1101185, NSF CAREER
award DMS-0812005 and by
a USA-Israeli BSF grant.}
}

\maketitle

\begin{abstract}
We describe two constructions of (very) dense graphs which are edge
disjoint unions of large {\em induced}  matchings. The first construction
exhibits graphs on $N$ vertices with ${N \choose 2}-o(N^2)$ edges,
which can be decomposed into pairwise disjoint induced matchings,
each of size $N^{1-o(1)}$. The second construction provides a
covering of all edges of the complete graph  $K_N$ by two graphs, each
being the edge disjoint union of at most $N^{2-\delta}$ induced
matchings, where $\delta > 0.058$. This disproves (in a strong form) a conjecture of
Meshulam, substantially improves a result of Birk, Linial and Meshulam on
communicating over a shared channel, and (slightly)
extends the analysis of H{\aa}stad and Wigderson of the 
graph test of Samorodnitsky and Trevisan for linearity. Additionally, our constructions
settle a combinatorial question of Vempala regarding a candidate rounding scheme for the directed
Steiner tree problem. 
\end{abstract}

\setcounter{page}{0} \thispagestyle{empty}
\newpage

\section{Introduction}

\subsection{Background}

Dense graphs consisting of large pairwise edge disjoint induced matchings
have found several applications in Combinatorics, Complexity Theory and 
Information Theory. Call a graph $G=(V,E)$ an
$(r,t)$-Ruzsa-Szemer\'edi graph ($(r,t)$-RS graph, for short)  if
its set of edges consists of $t$ pairwise disjoint induced
matchings, each of size $r$. The total number of edges of such a
graph is clearly $rt$. Graphs of this type are useful 
when both $r$ and
$t$ are relatively large  as a function of the number of vertices
$N$.  There are several known interesting constructions, relying on 
a variety of techniques. 

The first surprising construction was given by 
Ruzsa and Szemer\'edi \cite{RS}, who
applied a result of Behrend \cite{Be} about the existence 
of dense subsets of 
$\{1,2,...,\Theta(N)\}$  containing no 3-term arithmetic progressions
to prove that there are $(r,t)$-RS graphs on $N$ vertices with
$r=\frac{N}{e^{O(\sqrt {\log N})}}$ and $t=N/3$. They applied this
construction, together with the regularity lemma of Szemer\'edi
\cite{Sz}, to settle an extremal problem of Brown, Erd\H{o}s and 
S\'os \cite{BES1,BES2}, showing that the maximum possible number of edges
in a $3$-uniform hypergraph on $N$ vertices which contains
no $6$ vertices spanning at least $3$ edges is bigger than
$N^{2-\epsilon}$ and smaller than $\epsilon N^2$, for any $\epsilon
>0$, provided $N> N_0(\epsilon)$. See also \cite{EFR},\cite{AS3}
for more details about this problem, its extensions,  
and their connection to $(r,t)$-RS
graphs.

Note that the above construction provides graphs with $N$ vertices
and $\frac{N^2}{e^{O(\sqrt {\log N})}}$ edges, that is, rather
dense graphs, but still  ones in which the number of edges is
$o(N^2)$. These graphs and some appropriate variants have
been used by the first author in
\cite{Al}, to show that the problem of testing $H$-freeness in
graphs requires a super-polynomial (in $1/\epsilon$)  number of queries 
if and only if $H$ is not bipartite. The proof for one sided error
algorithms is given in \cite{Al}, and an extension  
for two sided algorithms is described in \cite{AS1}. A similar
application of these graphs
for testing induced $H$-freeness appears in \cite{AS2}, and yet
another very recent application  showing that testing graph-perfectness
requires a super polynomial number of queries appears in \cite{AF}.

The above graphs have also been applied by 
H{\aa}stad and Wigderson \cite{HW} to give an improved analysis
of the graph test of Samorodnitsky and Trevisan for linearity and for
PCP with low amortized complexity \cite{ST}. 

Another construction of $(r,t)$-RS graphs on $N$ vertices,
with $r=N/3-o(N)$ and $t=N^{\Omega(1/\log \log N)}$ was given by
Fischer et. al. in \cite{FNRRS}. Note that the matchings here are
of linear size, but their number is much smaller than in the 
original construction of Ruzsa and Szemer\'edi. The construction
here is combinatorial, and Fischer et. al. use these graphs to establish an
$N^{\Omega(1/\log \log N)}$ lower bound for  testing monotonicity
in general posets.

Yet another construction was obtained by Birk, Linial and Meshulam
\cite{BLM}, and in an improved form by Meshulam 
\cite{Me}. For the application in \cite{BLM} it is crucial to
obtain graphs with positive density. Indeed, the graphs here
are $(r,t)$-graphs on $N$ vertices with 
$r=(\log N)^{\Omega(\log \log N/ (\log \log \log N)^2)}$
and $t$ roughly $\frac{N^2}{24r}$. Thus, their number of edges 
is about $N^2/24$.
The method here relies on a construction of low degree
representation of the OR function, due to Barrington, Beigel and
Rudich \cite{BBR}. 
The application in \cite{BLM} is in Information Theory, the
graphs are applied to design an efficient deterministic
scheduling scheme for communicating over a shared
directional multichannel. 

Interestingly, none of these constructions address the question of whether
or not an $(r,t)$-RS graph can simultaneously have positive density and
yet be an edge disjoint union of polynomially large induced matchings.
This range of parameters is important for some applications - especially
ones in which there is a tradeoff between the number of missing edges
and the number of induced matchings needed to cover the graph. Indeed,
Meshulam \cite{Me} conjecture that there are no such graphs. We are able
to disprove this conjecture in the strongest possible sense: The density
of our construction is $1-o(1)$ and yet $r$ is nearly linear in $N$.
We also give a number of applications of our constructions.

\subsection{Our Results}

We construct $(r,t)$-RS graphs on $N$ vertices with $rt=(1-o(1)){N \choose
2}$, and $r=N^{1-o(1)}$. Thus, not only can we have graphs with
positive edge density which are edge disjoint union of induced matchings
of size $N^{\Omega(1)}$, we can in  fact have the edge density 
$1-o(1)$, where the size of each matching is $N^{1-o(1)}$. We also
describe another construction of a partition of the complete graph
$K_N$ into two subgraphs, each being a union of at most
$N^{2-\delta}$ induced matchings, where $\delta > 0.058$. 
The main difference between the new constructions presented here
and the previous ones mentioned above, 
is that the graphs  constructed here are
of density $1-o(1)$, that is, almost all edges of the complete
graph $K_N$ are covered, and yet all these edges can be partitioned
into large pairwise disjoint induced matchings. This surprising
property  turns out to be useful in various applications.

Our first construction is geometric, and is
inspired by the recent work of Fox and Loh \cite{FL} on dense
graphs in which every edge is contained in at least one triangle
and yet no edge is contained in too many triangles. The
construction follows the basic approach of Fox and Loh (slightly modified
according to the remark of the first author, mentioned at the end
of \cite{FL}) with different parameters. An additional (simple) argument
is required in decomposing sparse graphs into not too many induced matchings.
Our second construction applies some basic tools from Coding Theory. Also we make
us of the regularity lemma and some combinatorial and entropy based 
techniques to prove lower bounds for these questions. 

It is worth noting that a general result of Frankl and F\"uredi
\cite{FF}, implies that for any {\em fixed} $r$, there are
$(r,t)$-RS graphs $G$ on $N$ vertices with $rt=(1-o(1)) { N \choose
2}$. This is proved by choosing the non-edges of $G$ randomly
and by applying the nibble technique to obtain the existence
of the desired matchings. Using this method,  however, the induced
matchings obtained  are of constant size, whereas we are interested,
crucially, in large
matchings.  The techniques of \cite{FF} cannot provide induced 
matchings of size exceeding $\Theta( \log N)$.

We apply our results to significantly improve the application in \cite{BLM}. 
As mentioned earlier, Birk, Linial and Meshulam construct $(r,t)$-graphs on $N$ vertices with 
$r=(\log N)^{\Omega(\log \log N/ (\log \log \log N)^2)}$
and $rt$ roughly $\frac{N^2}{24}$. The authors then use these graphs to design
a communication protocol over a shared directional multi-channel -- it is critical for the
application that these graphs have positive density. The communication protocol
based on these graphs achieves a round complexity of $O(\frac{N^2}{r})$ and this is
a slightly better than poly-logarithmic improvement over the naive protocol for bus-based
architectures. 

 We can use 
 our construction to achieve a round complexity 
of $O(N^{2 - \delta})$ over a shared directional multi-channel.
This is the first such protocol that is a polynomial improvement over the naive protocol. 
 We can accomplish this using just two receivers per station (this 
corresponds to a 
 partition of the edges of a complete bipartite graph into two graphs that can be decomposed into large
 induced matchings). In case we are allowed $C = C(\epsilon)$ 
 receivers per station, 
we can achieve a round complexity that 
is $O(N^{1 + \epsilon})$ for any $\epsilon > 0$. 
 Hence, previous protocols required nearly a quadratic number of rounds, and our protocols require
 only a nearly-linear number of rounds. 
 
 Our constructions also disprove the recent conjecture of Meshulam. 
Moreover, we can achieve a density approaching one while simultaneously being able to decompose the graph into
at most a nearly-linear number of induced 
 matchings.

Besides their applications to the problems of \cite{BLM} and \cite{Me}, our
constructions can be plugged in the result of \cite{HW},
extending it to a new range of the parameters that may be of
interest. Lastly, we also answer a question of Vempala \cite{Ve}, showing
that a certain rounding scheme for the directed Steiner tree
problem is not effective.

The rest of this paper is organized as follows. In the next section
we describe the two new constructions. Lower bounds showing that these are
not far from being tight are given in Section 3. In Section 4 we describe
the applications of these graphs. The final Section 5 contains 
some concluding remarks and open problems.

\section{Constructions}

\subsection{A Geometric Construction}\label{sec:geom}

Here we construct nearly-complete graphs that can be
covered by an almost linear number of induced matchings. These graphs will
be based on a geometric construction inspired by a recent construction of
Fox and Loh \cite{FL}. 

We first describe a graph $G = (V, E)$ and then prove 
that it can be slightly modified to yield a nearly complete 
$(r, t)$-RS graph. 
Set $V = [C]^n$ for some constant $C$ to be chosen later. Let
$N = C^n$ be the number of vertices. Each vertex $x \in V$
will be interpreted as an integer vector
in $n$ dimensions with coordinates in $[C]=\{1,2,\ldots ,C\}$, 
where for technical reasons it is convenient to assume
that $n$ is even. Let $\mu = E_{x, y}[\| x - y \|_2^2]$ where $x$ and
$y$ are sampled uniformly at random from $V$. We could compute $\mu$,
but we will not need this exact value.

Next, we describe the set $E$ of edges. A pair of vertices $x$ and $y$ 
are adjacent if and only if $$ \Big | \|x - y \|_2^2 - \mu \Big | \leq n.$$

This condition implies that the number of missing edges is small, 
by a standard application of Hoeffding's Inequality:

\begin{claim}
${N \choose 2} - |E| \leq {N \choose 2} 2e^{-n/2C^4}$
\end{claim}

\begin{proof}
The quantity $\|x - y\|_2^2 = \sum_{i = 1}^n (x_i - y_i)^2$ and hence
is the sum of independent random variables (when $x$ and $y$ are chosen
uniformly at random from $V$). Each variable is bounded in the range $[0,
C^2]$ and hence we can apply Hoeffding's Inequality 
to obtain 
$$Pr[|\|x - y \|_2^2 - \mu | > n] \leq 2e^{-n/2C^4}$$ and this implies
the Claim.
\end{proof}

As a first step, we will cover all the edges of $G$ by a linear number
of induced subgraphs of small (but super-constant) maximum degree. We
will then use this covering to obtain a covering via an almost linear
number of induced matchings. Next, we describe the (preliminary) induced
subgraphs that we use to cover $G$.

We will define one subgraph $G_z$ for each $z \in V$. Let $V_z$ (the
vertex set of $G_z$) 
be $$V_z = \Big \{ x \in V ~:~ |\| x - z\|_2^2 -
\mu/4 | \leq 3n/4 \Big \}.
$$ 
The subgraph $G_z$ is the induced graph on $V_z$. 
First, we prove that these subgraphs $G_z$ do indeed cover the edges of $G$:

\begin{lemma}
Let $n \geq 2C$. 
For all $(x, y) \in E$, there is a $z$ such that $x, y \in V_z$
\end{lemma}

\begin{proof}
First we establish a simple 
Claim that will help us choose an appropriate $z$:

\begin{claim}~\label{claim:sign}
Let $a$ be a vector in which the absolute value of each entry is at most
$C$. Then there is a vector $w$ where each entry is $\pm 1/2$ such that $|
(a, w) | = | \sum_{i=1}^n a_i w_i | \leq C/2 \leq n/4$
\end{claim}

\begin{proof}
We can prove this by induction by considering the partial sum 
$  \sum_{i=1}^r a_i w_i $ which we will assume is at most $C/2$ 
in absolute value. We can choose $w_{r+1}$ 
so that $a_{r+1}w_{r+1}$ has 
the opposite sign of this partial sum and this implies that the 
partial sum $  \sum_{i=1}^{r+1} a_i w_i $  is also 
at most $C/2$ in absolute value (although the sign may have changed). 
This completes the proof of the claim. 
\end{proof}

Let $a$ be a vector defined as follows: 
if $y_i - x_i$ is even, set $a_i = 0$
and otherwise set $a_i = y_i - x_i$. We can apply  
Claim~\ref{claim:sign} to $a$,
but furthermore change the values of $w$ to be zero on indices on which
$a$ is zero. Then $|(a, w)|$ is still at most $C/2$, and $w_i$ is $\pm 1/2$
whenever $a_i$ is non-zero, and zero whenever $a$ is zero. Set $z =
\frac{y + x}{2} + w$. Note that $z \in V$ because whenever $\frac{y_i
+ x_i}{2}$ is not an integer, $a_i$ must be non-zero and hence $w_i$
is $\pm 1/2$ and alternatively 
whenever $\frac{y_i + x_i}{2}$ is an integer, $a_i$
is zero and hence $w_i$ is zero. Consider the quantity 
$$\|z - x\|_2^2 = \|\frac{y -
x}{2} + w\|_2^2 = \frac{1}{4} \|y - x \|_2^2 + (y - x, w) + \|w\|_2^2$$ 
Since
$x$ and $y$ are adjacent in $G$, we have that 
$|~\frac{1}{4} \|y - x \|_2^2 - \mu / 4 ~| \leq n/4$. 
Also, $\|w\|_2^2 \leq n/4$. Finally, $(y - x, w) = (a,
w)$ since $w$ is zero iff $a$ is zero. Hence 
$|~\|z - x\|_2^2-\mu/4~| \leq n/4+C/2+n/4 \leq 3n/4$, 
and
an identical argument holds for bounding $|~\|z - y \|_2^2-\mu/4~|$. 
Thus $x, y \in V_z$
and the edge $(x, y)$ is covered by some induced subgraph $G_z$.
\end{proof}

Next we establish that the maximum degree of any induced subgraph $G_z$
is not too large:

\begin{lemma}
For all $z \in V$, the maximum degree of $G_z$ is at most $(10.5)^n$. 
\end{lemma}

\begin{proof}
Let $x \in V_z$ and consider any neighbor $y$ of $x$  that satisfies
$y \in V_z$. We first establish that $x$ and $y$ are close 
to being antipodal
in the ball centered at $z$,
and hence we can bound the number of neighbors of $x$ in $V_z$ by bounding
the number of points of $V$ in some small 
spherical cap around the antipodal
point to $x$.

Define the antipodal point $x' = 2z - x$, this is 
the antipodal to $x$ with respect to the ball centered at $z$.

Consider the parallelogram $(x, z, y, x+y-z)$. By the Parallelogram
Law, the sum of the squares of the four side lengths equals the
sum of the squares of the lengths of the two diagonals. Therefore
we obtain 
$$\|x - y \|_2^2 + \|x + y - 2z \|_2^2 = 2 \|x - z\|_2^2 +
2 \|y - z \|_2^2$$ Using the definition of $x'$, this gives $$\|x -
y \|_2^2 + \|y - x' \|_2^2 = 2 \|x - z\|_2^2 + 2 \|y - z \|_2^2$$

Hence $\|y - x'\|_2^2 = 2 \|x - z\|_2^2 + 2 \|y - z \|_2^2 - \|x - y
\|_2^2$ and as both $\|x - z\|_2^2$ and $\|y - z \|_2^2$ are approximately
$\mu/4$ (since $x, y \in V_z$) and $\|x-y\|_2^2$ is approximately $\mu$
because $x$ is adjacent to $y$ in $G$, this implies that 
$\|y - x'\|_2^2 \leq 4n$.
Therefore we can bound the degree
of $x$ in $G_z$ by the number
lattice points in a ball of radius $2 \sqrt n$ (centered at the lattice 
point $x'$.)  The unit $n$-dimensional cubes centered  at these 
lattice points are pairwise disjoint, each has volume $1$, and they are
all contained in a ball of radius $2 \sqrt n+0.5 \sqrt n=2.5 \sqrt n$.
Therefore, the number of these points does not exceed the volume of an
$n$ dimensional ball of radius $2.5 \sqrt n$. Since $n$ is even, the 
volume of this ball is 
$$
\frac{\pi^{n/2} (2.5 \sqrt n)^n}{(n/2)!} 
< (2 \pi e)^{n/2} \frac{(2.5 \sqrt n)^n}{n^{n/2}} =(2.5 \cdot 
\sqrt {2 \pi e })^n < 10.5 ^n,
$$
where here we have used the fact that $b! > (b/e)^b$ for any positive 
integer $b$. This completes the proof of the lemma.
\end{proof}

\begin{lemma}
Let $H$ be a graph with maximum degree $d$. 
Then $H$ can be covered by $O(d^2)$ induced matchings.
\end{lemma}

\begin{proof}
Call two edges $e_1, e_2$ of $H$ in conflict if they either 
share a common end  or
there is an
edge in $H$ connecting an endpoint of $e_1$ to an endpoint of $e_2$.
It is clear that any edge $e$ of $H$ 
can be in 
conflict with at most $2d-2+(2d-2)(d-1)<2d^2$ other edges of $H$.

Thus we can initialize each member of a set of $2d^2$ 
induced matchings $M_i$ to be the empty set, and for each edge
$e$ of $H$ in its turn, add $e$ to the first $M_i$ for which $e$ is not
in conflict with any edge currently in $M_i$. Since $e$ is in conflict
with less than $2d^2$ edges, it can added to some $M_i$. We can
continue this procedure, obtaining a set of less than $2d^2$ induced 
matchings covering all  edges of $H$.
\end{proof}

It follows that we can decompose the edges of each induced 
subgraph $G_z$ into at most $O(d^2) \leq O((10.5)^{2n})$ 
induced matchings,
and this yields a decomposition of $G$ 
into $O(N d^2)$ induced matchings. 
These matchings can additionally
be made edge-disjoint, since, if any edge is multiply covered 
we can remove it from all but one of the
induced matchings (and the result is still an induced matching). 
We have thus proved the following.

\begin{theorem}
For every $n, C$ with $n \geq 2C$, $n$ even, there is a graph $G$ on 
$N = C^n$ vertices that is missing at most $ N^g$ edges, 
for $$g = 2 - \frac{ 1 }{2C^4 \ln C} + o(1)$$ and 
can be covered by $N^{f}$ disjoint induced matchings, 
where $$f =1 + \frac{2 \ln 10.5}{\ln C} + o(1)$$. 
\end{theorem}

Hence, for any $\epsilon>0$ we can construct a graph $G$ on $N$ vertices 
missing 
at most $N^{2 - \delta}$ edges for  $\delta = \delta(\epsilon)=
e^{-O(1/\epsilon)}$,
that can be covered by $N^{1+\epsilon}$ 
pairwise disjoint induced matchings. 
Note that the number of matchings is nearly linear.
Note also that by splitting each of these matchings $M$ into
$\lfloor |M|/r \rfloor$ pairwise 
disjoint matchings, each of size exactly $r=N^{1-\epsilon-\delta}$,
omitting the remaining $|M|-r \lfloor |M|/r \rfloor <r$ edges, we
get an $(r,t)$-RS graph, where $r=N^{1-\epsilon-\delta}$ and
the number of missing edges  is at most $2N^{2-\delta}$. As $\epsilon $
(and hence $\delta<\epsilon$) can be chosen to be 
arbitrarily small, this gives, with the right choice of parameters,
an $(r,t)$-RS graph on $N$ vertices, with $r=N^{1-o(1)}$ and
$rt={N \choose 2} -o(N^2)$. 

\subsection{A Construction Using Error Correcting Codes}
\label{sec:ecc}

Here we construct nearly-complete graphs with large
induced matchings using error correcting codes. 
These constructions will be incomparable to those
in the previous section - the number of missing edges 
will be much smaller (in fact, the number of missing
edges can be made asymptotically optimal as 
we will demonstrate in Section~\ref{sec:recon}), but the price  we
pay is that 
the average size of an induced matching will only be 
a small power of $N$ as opposed to nearly-linear. 
As we will show, 
the construction in this section will be 
better tailored to the application in \cite{BLM} (at least for some
values of the relevant parameters) than the construction
of the
previous section. 

Throughout this section,
we will use codes over the binary alphabet as well as over a bigger
alphabet. Let $d_H(x, y)$ be the Hamming
distance between two binary strings $x$ and $y$ (of the same length).  The
Hamming weight of a binary string $x$ is the number of non-zero entries
- or equivalently the Hamming distance to the all zeros vector. We can
similarly define the Hamming distance $d_H(x,y)$ 
between two vectors $x$ and $y$ over a larger
alphabet $[C]$ as the number of indices where these vectors disagree.

\begin{definition}
A $[n, k, d]$ linear 
code $\calC$ is a subspace consisting of $2^k$ length $n$ binary vectors
such that for all $x, y \in \calC$ and $x \neq y$, $d_H(x, y) \geq
d$. We will call $n$ the encoding length, $k$ the dimension, and $d$
the distance of the code.
\end{definition}

An $n \times k$ matrix $A$ of full column rank 
over $GF(2)$ is the generating matrix 
of a code of dimension $k$ and
length $n$ consisting of all linear combinations of its columns.
The distance of this code is exactly the minimum Hamming
weight of any non-zero code word. Throughout this section we will make
use of a particular type of code:

\begin{definition}
Call a linear code $\calC$ 
proper if the all ones vector is a codeword. 
\end{definition}

It is well-known that there are linear codes that achieve the
Gilbert-Varshamov Bound. In fact, proper codes also achieve
this bound:

\begin{lemma}
If $ \sum_{i = 0}^d {n  \choose i} < 2^{n-k}$, 
then there is a proper $[n, k, d]$ code.  Thus, there is such a code
in which $k=(1-H(d/n))n$, where $H(x)=-x \log_2 x -(1-x) \log_2 (1-x)$
is the binary entropy function.
\end{lemma}

\begin{proof}
We can define a length $n$ code by choosing an $(n - k) \times n$ parity
check matrix as follows: for each of the first $n-1$ columns, choose
each vector uniformly at random. Choose the last column to be the parity
of the preceding $n-1$ columns. Let the matrix be $B$. Then the code
$\calC$ is defined as $\calC= \{x \in \{0, 1\}^n | B x = \vec{0}\}$. Clearly
$\vec{1} \in \calC$ by construction. Since this code is linear, the
minimum distance is exactly the minimum Hamming weight of any non-zero
codeword. This quantity is exactly the smallest number of columns of $B$
that sum to the all zeros vector.

\begin{claim}
For any fixed set $S$ of columns of $B$, the probability that the sum is the all zeros vector is exactly $2^{-(n-k)}$
\end{claim}

\begin{proof}
If this set $S$ does not contain the last column, then the sum of 
the columns is distributed uniformly on $\{0, 1\}^{n-k}$. 
If the set $S$ does contain the last column, then the sum of 
the columns is exactly the sum of the columns not in the 
set - i.e. $[n] - S$ - and hence is also distributed 
uniformly on $\{0, 1\}^{n-k}$.
\end{proof}

So the probability that any set of at most $d$ columns sums to the 
all zero vector is at most $2^{-(n-k)} \sum_{i = 0}^d {n  \choose i}$, 
and if $ \sum_{i = 0}^d {n  \choose i} < 2^{n-k}$ there is a 
parity check matrix $B$ so that the code $\calC$ has distance 
at least $d+1>d$. This code has dimension at least $k$ because 
there are $n-k$ constraints imposed by the parity check matrix. 
If these constraints are linearly independent, then the code has 
dimension exactly $k$. If these constraints are not linearly 
independent, we can add additional constraints until the code has 
dimension exactly $k$ and the distance of the 
code cannot decrease as we add these constraints. 
\end{proof}

\begin{claim}
Let $A$ be the generating matrix for a proper $[n, k, d]$ code $\calC$
with $d>1$. 
Then deleting any row of $A$ results in a 
generating matrix $A'$ for a proper $[n-1, k, d-1]$ code. 
\end{claim}

\begin{proof}
Note that $\calC' = \{A'x | x\in \{0, 1\}^k\}$ and 
hence $\calC'$ (defined by the generating matrix $A'$) has 
dimension $k$, as no nontrivial linear combination of the columns of
$A'$ can be the zero vector, by the assumption $d>1$. 
The all ones vector is still a codeword since 
$Ax = \vec{1}$ implies that $A'x = \vec{1}$. Finally, the 
minimum distance of $\calC'$ is the minimum Hamming weight of any 
non-zero code word, and the Hamming weight of any 
codeword in $\calC$ decreases by at most one by deleting any index. 
\end{proof}

Throughout this section let $\calC = \calC_n$ be an $[n, k, d]$ code, 
and let $\calC_{n-1}, \calC_{n-2}, ... \calC_{n-d+1}$ be proper 
$[n-1, k, d-1]$, $[n-2, k, d-2]$, ... and $[n-d+1, k, 1]$ 
codes, respectively. 

Next, we define a graph $G = (V, E)$  that will be the focus of 
this section. Let $V = [C]^n$ and set $N = |V| = C^n$. 
Consider two vertices $a, b \in V$, 
where $a = (a_1, a_2, ... a_n)$ and $b = (b_1, b_2, ... b_n)$ 
for $a_i, b_i \in [C]$. There is an edge between $a$ and $b$ if and 
only if $d_H(a, b) = \sum_{i = 1}^n 1_{a_i \neq b_i} > n-d$. 

It is easy to count the number of missing edges. Indeed, in the complement
of $G$ each vertex $a$ is connected to all vertices $b$ so that
$a_i=b_i$ for at least $d$ indices $i$. As the number of missing
edges is half the sum of degrees in the complement this gives:

\begin{claim}
$${N \choose 2} - |E| \leq \frac{1}{2}C^N \sum_{i = d}^n 
{n \choose i} (C-1)^{n-i}. $$
\end{claim}

\begin{lemma}
If $\frac{d}{n} \geq \frac{2}{C-1}$ then 
$$\frac{1}{2} C^n \sum_{i = d}^n {n \choose i} (C-1)^{n-i} \leq  
{n \choose d} C^n (C-1)^{n - d}.$$
\end{lemma}

\begin{proof}
Using the inequality 
${n \choose i} \leq (n/d)^{i - d} {n \choose d}$ we obtain a bound 
$$\sum_{i = d}^n {n \choose i} (C-1) ^{n-i} 
\leq {n \choose d} (C-1)^{n-d} \sum_{j = 0}^{n-d} (n/d)^j(C-1)^{-j }$$
which implies the Lemma. 
\end{proof}

Hence the number of missing edges in $G$ is at  most $ N^{e}$ 
for $e = 1 + \frac{ H(d/n) + (1- d/n)\log_2 (C-1) }{\log_2 C} + o(1)$. 
\vspace{0.3cm}

Next, we describe the induced matchings that are used to cover the 
edges in $G$. In order to do so, we will define an equivalence 
relation over edges of $G$. In particular, this will be an 
equivalence relation over ordered pairs $(a, b)$, 
where $a = (a_1, a_2, ... a_n)$ and $b=(b_1, b_2, ... b_n)$, 
under the condition that $d_H(a, b) > n-d$. 

\begin{definition}
Let $S \subset [n]$, $|S| = r$ and let $(a, b)$ be a pair of 
vertices in $V$ where $S=\{i | a_i=b_i\}$.
Let $x$ be a $\{0, 1\}^{n-r}$ vector. Let 
$[n] - S = \{i_1, i_2, ... i_{n-r}\}$ and $i_1 < i_2, ... < i_{n-r}$. 
Then the $x$-flip of $(a, b)$ is a pair $(c, d)$ such that 
for all $i \in S$, $c_i = a_i= b_i= d_i$ and for all 
$i = i_j \notin S$ (i.e. $i$ is the $j^{th}$ smallest 
index not in $S$), $c_i = a_i, d_i = b_i$ 
if $x_j = 0$ and otherwise $c_i = b_i, d_i = a_i$. 
\end{definition}

Informally, the $n-r$ indices not in $S$ are mapped in order to 
the $n-r$ bits in $x$ and the corresponding locations in $a$ 
and $b$ are swapped if and only if the corresponding 
bit of $x$ is one. 

\begin{definition}
We will define a pair $(a, b) \sim (a', b')$ iff 
$S = \{i | a_i = b_i\} = S' = \{i | a'_i = b'_i\}$, 
$|S| < d$ and furthermore there is an $x \in \calC$ 
such that $(a', b')$ is the $x$-flip of $(a, b)$. 
\end{definition}

Next we will establish that this relation $\sim$ 
is indeed an equivalence relation, and that it is actually a relation
on unordered pairs, that is $(a, b) \sim (b, a)$ for all $a,b$:

\begin{claim}
$(a, b) \sim (b, a)$
\end{claim}

This follows because the code $\calC_{n-r}$ is proper 
(for all $r < d$), and hence the all ones vector $\vec{1}$ lies in
$\calC_{n-r}$ 
and $(b, a)$ is the $\vec{1}$-flip of $(a, b)$. 

\begin{claim}
$(a, b) \sim (c, d)$ iff $(c, d) \sim (a, b)$
\end{claim}

\begin{proof}
By symmetry we only need to establish one direction. 
Suppose $(a, b) \sim (c, d)$. Then $S = \{i | a_i = b_i\} 
= S' = \{i | c_i = d_i\}$. Let $(c, d)$ be an $x$-flip of 
$(a, b)$ (where $x \in \calC_{n-r}$). Then $(a, b)$ is also 
the $x$-flip of $(c, d)$. 
\end{proof}

\begin{claim}
$(a, b) \sim (c, d)$ and $(c, d) \sim (e, f)$ implies $(a, b) \sim (e, f)$
\end{claim}

\begin{proof}
Again note that $S = \{i | a_i = b_i\} = S' = \{i | c_i = d_i\} 
= S'' = \{i | e_i = f_i\}$. Let $x, y \in \calC_{n-r}$ be such 
that $(c, d)$ is the $x$-flip of $(a, b)$ and $(e, f)$ is 
the $y$-flip of $(c, d)$. Then $x + y \in \calC_{n-r}$ 
since the code is linear, and $(e, f)$ is the $x+y$-flip of $(a, b)$. 
\end{proof}

This immediately implies:

\begin{lemma}
The relation $\sim$ is an equivalence relation over unordered 
pairs $(a, b)$ which have Hamming distance $> n -d$. 
\end{lemma}

Since each code $\calC_{n-r}$ (for $r < d$) has dimension 
$k$, each equivalence class has size exactly $2^k$. 

\begin{lemma}
Each equivalence class is an induced matching consisting of
$2^{k-1}$ edges.
\end{lemma}

\begin{proof}
Consider two edges $(a, b)$ and $(e, f)$ in the same equivalence 
class. Let $S = \{i | a_i = b_i\} = \{i | e_i = f_i\}$ 
where $|S| = r ~ (<d)$. Let $(e, f)$ be the $x$-flip of 
$(a, b)$ for $x \in \calC_{n-r}$. Since the code $\calC_{n-r}$ 
has distance at least $d-r$, the Hamming weight of $x$ is at 
least $d-r$. Consider the Hamming distance between $a$ and $f$. 
Each index $i \in S$ is an index at which $a$ and $f$ agree 
(i.e. $a_i = f_i$). Furthermore, there is a bijection between 
indices in $x$ that are set to one and indices outside of the set 
$S$,  for which $a$ and $f$ agree. So the vectors 
$a$ and $f$ agree on at least $r + (d - r) = d$ indices, 
and hence $af$ is not an edge in $G$.  Since $(e,f)$ and
$(f,e)$ are in the same equivalence class the above argument also
shows that $ae$,$be$ and $bf$ are nonedges.
\end{proof}

If we use one induced matching for each equivalence class, 
then each edge in $G$ is covered exactly once and hence the 
number of induced matchings needed to cover $G$ 
is $\frac{|E|}{2^{k-1}} \leq \frac{N^2}{2^k}$. 

\begin{theorem}~\label{thm:ecccon}
For every $n, d, C$ such that $\frac{d}{n} \geq \frac{2}{C-1}$ ,
there is a graph $G$ on $N = C^n$ vertices that is missing at most 
$ N^e$ edges, for 
$$
e = 1 + \frac{ H(d/n) + (1- d/n)\log_2 (C-1) }
{\log_2 C} + o(1)
$$ 
and can be covered by $N^{f}$ disjoint induced matchings, 
where $$f =2-  \frac{1 - H(d/n)}{\log_2 C} + o(1)$$. 
\end{theorem}

In particular, for any $\eps > 0$, there is a graph $G$ on $N$ vertices
missing at most $N^{3/2 + \eps}$ edges that can be covered 
by $N^{2 - c \eps^3}$ induced matchings.  This is obtained by choosing
$C$ for which $\log_2 C=\Theta(1/\epsilon)$ and 
$d/n=\frac{1}{2}-\Theta(\epsilon)$.

Also we can choose $C = 34$ and $d = 0.19 n$, in which case 
$e, f <  1.942$. Thus we can cover the edges of a complete 
graph on $2N$ vertices
by two graphs $G_1$ (set to $G$ with $N$ replaced by $2N$ 
in the above construction) and $G_2$ (set to the complement of $G$), 
where the number of induced matchings needed to cover the 
edges of $G_1$  is 
$O(N^{2 - \delta})$ for $\delta > 0.058$, and the same holds for $G_2$. 
For the applications we need that the above statement holds also for 
covering all edges of the complete bipartite graph $K_{N,N}$ by
two such graphs $G'_1$ and $G'_2$-this clearly follows by 
splitting the vertices of $K_{2N}$ arbitrarily into two equal classes
and by defining $G'_i$, for $i=1,2$,
to be the graph obtained from $G_i$ by keeping
only the edges that have one endpoint in each class.

\section{Limits}

\subsection{Triangle Removal Lemma}

The connection between the triangle removal lemma and the existence
of $(r,t)$-RS graphs is well known since the work of Ruzsa and 
Szemer\'edi, for completeness we include the argument.

\begin{prop}
\label{p91}
If there exists an $(r,t)$-RS graph  on $N$ vertices, then there
exists a graph on $N+t$ vertices with at least $3rt/2$ edges,
in which every edge is contained in exactly one triangle. Thus one has
to delete at least $rt/2$ edges to destroy all triangles and yet
the graph contains only $rt/2$ triangles.
\end{prop}

\begin{proof}
Let $G$ be an $(r,t)$-RS graph  on $N$ vertices. Then its number
of edges is $rt$ and hence, by a well known simple result, it contains 
a bipartite subgraph 
$G' = (U, V, E')$ with at least $rt/2$ edges. Clearly, these
edges can be covered by $t$ induced matchings $M_1, M_2, ... M_t$, and
we can  assume that these matchings are pairwise edge disjoint. 

For each matching $M_i$, add an additional vertex
$w_i$ and connect $w_i$ to the endpoints of all edges in $M_i$. 
The resulting graph $H = (U, V, W, E_H)$ is tripartite, 
has $N+t$ vertices and contains $|E'|$ triangles. The critical 
property of this construction is that 
each edge of $H$ is in a unique triangle. 
Indeed, there is a natural set of  $|E'|$ triangles in $H$ - 
each such triangle is specified by an edge $(u, v) \in E'$ and if 
this edge is contained in the matching $M_i$, this edge is mapped 
to the triangle $(u, v, w_i)$ in $H$. 
There are in fact no other triangles in $H$: Let $T = (a, b, c)$ be 
a triangle in $H$. Since $H$ is tripartite, there must be exactly 
one vertex from each set $U, V$ and $W$ in the set 
$a, b, c$. Suppose that $a \in U$ and $b \in V$. 
Then let $M_i$ be the unique matching containing the edge $(a, b)$. 
Suppose $c = w_j \neq w_i$. This implies that the matching $M_j$ 
covers both vertices $a$ and $b$ but does not contain the edge 
$(a, b)$, and hence $M_j$ is not an induced matching, contradiction. 
This completes the proof.
\end{proof}

The triangle removal lemma  of \cite{RS}, which is one of the early
major applications of the regularity lemma, asserts that for any
$\epsilon>0$ there is a $\delta=\delta(\epsilon)>0$ so that
for $N>N(\epsilon)$ any graph on $N$ vertices from which one has
to delete at least  $\epsilon N^2$ edges to destroy all triangles
contains at least $\delta N^3$ triangles. This and the above  
proposition implies that there are no $(r,t)$-RS graphs on $N$ vertices
with $r=\Omega(N)$ and $t=\Omega(N)$. The original proof of \cite{RS}
provides a rather poor quantitative  relation between $\epsilon$ and
$\delta$, but the improved recent proof of Fox \cite{F} supplies 
better estimates (which are still very far from the known constructions).
If the number of vertices is $N$ and the graph is a pairwise disjoint union
of $t$ induced matching, each of size
$r=cN$, then  $t$ is at most $N/\log^{(x)} N$,
with $x = O(\log (1/c))$, where
$\log^{(x)} N$ denotes the $x$-fold iterated logarithm.  For more details,
see \cite{F}.

\subsection{Reconstruction Principle}\label{sec:recon}

Here we prove lower bounds on the number of edges that a graph 
must miss, if it can be covered by disjoint induced matchings 
of size $r$. These lower bounds establish that the results in 
Section~\ref{sec:ecc} are essentially tight for an important range
of the parameters. Indeed, as proved in that section
there are graphs on $N$ vertices missing 
$N^{3/2 + \epsilon}$ edges that can be covered by disjoint, 
induced matchings of polynomial size. Yet, as we show below,
any graph that can 
be covered by disjoint, induced matchings of size {\em two} 
or more must miss at least $N^{3/2}$ edges.  We describe two proofs,
The first is based on entropy considerations, and the second
is an elementary  combinatorial proof, that in fact yields a
somewhat stronger result, as it bounds the minimum degree in the
graph of missing edges. We believe, however, that both methods 
are interesting and each may have further applications.
We start with the entropy proof.

Let $G = (V, E)$ be a graph on $N$ vertices that can be covered by
disjoint induced matchings $M_1, M_2, ... M_t$ each of size $r \geq
2$. We will prove an upper bound on the number of edges $|E|$ based on
an application of the reconstruction principle (and through information
theoretic inequalities).

To this end, we define a random variable $A$ as follows: 

\begin{itemize}

\item Choose $M_i$ uniformly at random

\item Choose an ordered set of two distinct edges $e_1, e_2$ from $M_i$

\end{itemize}

Set $A = (e_1, e_2)$. Let $e_1 = (W, X)$ and $e_2 = (Y, Z)$. Here we
use upper-case letters to denote that each of these choices $W, X, Y$
and $Z$ is a random variable and we will use lower case letters to denote
specific choices of these random variables.

\begin{claim}\label{claim:lower}
$H(A) = \log |E| + \log (r-1)$
\end{claim}

\begin{proof}
Since we choose each matching $M_i$ uniformly at random, and each matching
is of the same size ($r$), 
the first edge $e_1$ is chosen uniformly at random
from the set $|E|$. Conditioned on the choice of $e_1$, the remaining edge
$e_2$ is chosen uniformly at random from the $r-1$ other edges in $M_i$.
\end{proof}

Let $d_v$ be the number of missing edges incident to $v \in V$. Let
$D_v$
be the set of non-neighbors of $v$, and let $f_v: D_v \rightarrow [d_v]$
be a function mapping each non-neighbor of $v$ to a unique integer in
the set $[d_v]$.

\begin{itemize}

\item Choose $A$ as above and let $e_1 = (w, x)$ and $e_2 = (y, z)$

\item Choose $S_1$ with probability $1/2$ to be either $w$ or $x$, and let $S_3$ be the opposite choice

\item Choose $S_2$ with probability $1/2$ to be either $y$ or $z$

\end{itemize}

We set the random variable $B = [s_1, f_{s_1}(s_2), f_{s_2}(s_3)]$.

\begin{lemma}\label{lemma:relation}
$H(B) \geq H(A)$
\end{lemma}

\begin{proof}
We prove that $A$ can be computed as a deterministic function of $B$, and then we apply the Chain Rule for entropy to prove the Lemma.

\begin{claim}
$A$ can be computed as a deterministic function of $B$
\end{claim}

\begin{proof}
Given $B$, we can compute $s_2$ using $s_1$ and $f_{s_1}(s_2)$, and 
using $s_2$ and $f_{s_2}(s_3)$ we can compute $s_3$. This in turn 
defines the edge $e_1 = (s_1, s_3)$ which uniquely 
determines $M_i$ since the set of matchings disjointly 
covers the edges in $G$. From $M_i$ and $s_2$, 
we can compute the remaining edge $e_2$: this is the unique 
edge incident to $s_2$ in the matching $M_i$. 
\end{proof}

The Chain Rule for entropy yields the expansion $H(B, A) = H(B) + H(A |
B)$, but $H(A | B) = 0$ because $A$ is a deterministic function of $B$. We
can alternatively expand $H(B, A)$ as $H(A) + H(B | A)$. Since $H(B |
A) \geq 0$ we get $H(B) = H(B, A) \geq H(A)$, as desired.
\end{proof}

Next, we give an upper bound for the entropy of $B$ (based on the number
of missing edges), and this combined with the Lemma above will imply a
contradiction if the number of missing edges is too small.

\begin{definition}
We will call a random variable $S$ on $V$ degree-uniform if $S$ chooses a random vertex proportional to the degree in $G$. 
\end{definition}

\begin{claim}\label{claim:unif}
$S_1$ and $S_2$ are degree-uniform random variables
\end{claim}

Note that these two random variables are not independent!

\vspace{0.5pc}

\begin{proof} We can choose the random variable $A$ by choosing an edge
uniformly at random from $E$, setting this edge to be $e_1$ and choosing
$e_2$ uniformly at random from the remaining edges in the matching $M_i$
that contains $e_1$. The distribution of $S_1$ in this sampling procedure
(for $A$) is clearly degree-uniform.

To prove the remainder of the Claim, we can slightly modify the sampling
procedure for $A$. We could instead choose an edge uniformly at random
from $E$ and set this edge to be $e_2$. Then choose an edge $e_1$
uniformly at random from the other edges in the matching $M_i$ that
contains $e_2$. This is an equivalent sampling procedure for generating
$A$, and from this procedure it is clear that $S_2$ is degree-uniform.
\end{proof}

Let $\bar{d}$ be the average degree in the complement of $G$. 

\begin{lemma}\label{lemma:upper}
$H(B) \leq \log N + 2 \log  \bar{d}$
\end{lemma}

\begin{proof}
We can decompose the random variable $B$ into $B_1 = s_1$, $B_2 =
f_{s_1}(s_2)$ and $B_3 = f_{s_2}(s_3)$. Again, using the Chain Rule
for entropy we obtain that $$H(B) = H(B_1) + H(B_2 | B_1) + H(B_3 |
B_2, B_1)$$

Since $S_2$ is a deterministic function of the random variables $B_2$ 
and $B_1$, we get $$H(B_3 | B_2, B_1) = H(B_3 | B_2, B_1, S_2) 
\leq H(B_3 | S_2)$$ 
We can upper bound $H(B_1)$ by $\log N$, and 
$$H(B_2 | B_1) = \sum_{s_1} Pr[S_1 = s_1] H(B_2 | S_1 = s_1) 
\leq \sum_{s_1} Pr[S_1 = s_1] \log d_{s_1}$$ 
Using Claim~\ref{claim:unif}, this is 
$$H(B_2 | B_1) = \sum_{s_1} \frac{N - 1 - d_{s_1}}{2 |E|} 
\log d_{s_1} \leq \sum_{s_1} \frac{N-1-\bar{d}}{2|E|} \log \bar{d}=
\log \bar{d},$$
where here we have used Jensen's Inequality and the concavity of 
the functions $\log x$ and $-x \log x$. 
An identical bound holds also for $H(B_3 | S_2)$ 
again using Claim~\ref{claim:unif} and thus
we get $H(B) \leq \log N + 2 \log \bar{d}$.
\end{proof}

We can apply Lemma~\ref{lemma:relation} and the bounds in
Lemma~\ref{lemma:upper} and Claim~\ref{claim:lower} to obtain the
following theorem:

\begin{theorem}
\label{t38}
Let $G = (V, E)$ be a graph on $N$ vertices that can be covered by 
disjoint induced matchings of size $r \geq 2$. 
Then the number of missing edges satisfies
$${ N \choose 2 } - |E| \geq (\frac{1}{2 \sqrt 2}-o(1)) 
r^{1/2} N^{3/2}.$$
\end{theorem}

We can apply a nearly identical argument in the case in 
which $G$ is a bipartite graph:

\begin{theorem}
\label{t99}
Let $G = (U, V, E)$ be a bipartite graph that can be covered 
by disjoint induced matchings of size $r \geq 3$. Then the number 
of missing edges satisfies 
$$|U| \times |V| - |E| \geq \Omega( r^{2/3} |U|^{2/3} |V|^{2/3}).$$
\end{theorem}

To prove this result, we choose $A'$ to be three distinct edges from
the matching $M_i$, and we use a length three path through pairs in
$U \times V$ that are not in $E$ to define the corresponding random
variable $B'$. Again, the proof uses information theoretic inequalities
and the fact that (if appropriately defined) $A'$ can be reconstructed
as a deterministic function of $B'$. It is worth noting
that for a bipartite
graph with $|U|=|V|=N$ and induced matchings of size $2$, there is
a simple construction missing only $N$ edges.
\vspace{0.5cm}

\noindent
We can also give a direct counting argument, which is somewhat stronger,
as it yields a lower bound on the minimum degree in the graph of 
missing edges. This counting argument proceeds by  estimating
the size of an appropriately defined set in two ways.
Let $G = (V, E)$ be an edge disjoint union of induced matchings 
$M_1, M_2, ... M_t$ each of size $r$. 
Again, let $d_v$ be the degree of $v$ in the complement of $G$. 

Set $$\cF = \{ (v, e) | v \in V, e \in E, v \notin e, \exists M_i \mbox{ s.t. } e \in M_i \mbox{ and } v \mbox{ is covered by } M_i \}$$

\begin{lemma}
$|\cF| \leq \sum_v \min \Big ( {d_v \choose 2}, (N -1 - d_v) (r-1) \Big )$
\end{lemma}

\begin{proof}
For each $v \in V$, $v$ appears in precisely 
$(N - 1 - d_v) (r-1)$ elements
of $\cF$ since $v$ belongs to exactly $(N -1 -d_v)$ matchings and for
each such matching there are exactly $r-1$ choices of an edge (in the
matching) that is not incident to $v$.

Alternatively, each $v \in V$ also appears in at most ${d_v \choose 2}$ 
elements of $\cF$: if $(v, e) \in \cF$ then $v$ must not be a 
neighbor of each endpoint of $e$ because the matching is induced. As
there are at most ${d_v \choose 2}$ choices of pairs of vertices 
that are not  neighbors of $v$, the desired result follows. 
\end{proof}

We can also compute the size of $\cF$ exactly:

\begin{lemma}
$|\cF| = \sum_v (r-1) (N -1 - d_v)$
\end{lemma}

\begin{proof}
For each edge $e \in E$, let $M_i$ be the corresponding matching that
covers $e$. There are exactly $2 (r-1)$ choices for a vertex (covered by
$M_i$) but not incident to $e$. Hence each edge $e$ appears in exactly $2
(r-1)$ elements of $\cF$. Thus
$$
|\cF| = 2 (r-1) \Big [ {N \choose 2} - \frac{1}{2} \sum_v d_v \Big ] 
= (r-1) \Big [N (N-1) - \sum_v d_v \Big ] = \sum_v (r-1) (N -1 - d_v)
$$
\end{proof}

Combining the two estimates for $|\cF|$ we conclude that
$$
\sum_v (r-1) (N -1 - d_v)
\leq \sum_v \min \Big ( {d_v \choose 2}, (N -1 - d_v) (r-1) \Big ).
$$
It follows that for every $v$, the minimum term in the right hand side
should be  $(N -1 - d_v) (r-1)$, since otherwise the inequality cannot 
hold. Therefore we have proved the following, which implies Theorem
\ref{t38}.

\begin{theorem}
If $G = (V, E)$ is a graph on $N$ vertices that is the disjoint 
union of induced matchings of size $r$, then the minimum degree $d$ 
in the complement of $G$ satisfies $${d \choose 2} \geq (r-1)(N-1-d).$$
\end{theorem}

The assertion of Theorem \ref{t99} can be also proved by a 
counting argument. We omit the details.

\section{Applications}

\subsection{Shared Communication Channels}

We apply our results to significantly improve the application in
\cite{BLM} of communicating over a shared directional multichannel.
Roughly, when communicating over a shared channel we want the edges
(corresponding to messages sent in some time step called a round) to
form an induced matching. Otherwise, a receiver will hear messages sent
from two different sources and the messages will appear garbled. Birk,
Linial and Meshulam construct graphs with positive density that can be
covered by roughly $\frac{N^2}{24 r}$ induced matchings where $r=(\log
N)^{\Omega(\log \log N/ (\log \log \log N)^2)}$.  The authors then
use these graphs to design a communication protocol for $N$ stations
over a shared directional multi-channel where the round complexity of
this protocol is $O(\frac{N^2}{r})$.  This is a slightly better than
poly-logarithmic improvement over the naive protocol for bus-based
architectures.

We can use 
our constructions to achieve a round complexity of 
$O(N^{2 - \delta})$ over a shared directional multi-channel.
This is the first such protocol that provides
a polynomial improvement over the naive protocol. 
We accomplish this using just one transmitter and two 
receivers per station. This corresponds to a 
partition of the edges of a complete bipartite graph into 
two graphs each of which can be decomposed into a small number of
induced matchings. If we allow $C = C(\epsilon)$ 
receivers per station, we can achieve a round complexity 
that is $O(N^{1 + \epsilon})$ for any $\epsilon > 0$ (here $N$ is
a trivial lower bound).
Hence, while previous protocols required a nearly 
quadratic number of rounds with a constant number of receivers
per station, our protocols require
only a nearly-linear number of rounds. 
 
Motivated by the application to communication over a shared channel, 
Meshulam \cite{Me} conjectured that any graph on $N$ vertices
with positive density 
cannot be covered by  $O(N^{2 - \delta})$ induced matchings.
The constructions presented in Section~\ref{sec:geom} 
and in Section~\ref{sec:ecc} disprove this conjecture in a strong sense. 

First we explain the model considered in \cite{BLM}. Roughly, the goal is
to design a good communication protocol using a small number of shared
communication channels. More precisely, suppose we have $N$ stations,
and each wants to send a (distinct) message to every other station. We
further assume that each message is (roughly) the same size. In this
context, it is often prohibitively expensive to build a point-to-point
communication channel from each station to every other one. Often, the
proposed solution is to use some form of a shared communication channel.
Indeed, the standard bus-based architecture connects all pairs of stations
using a single connection in such a way that only one message can be
sent on the channel per time step and hence a total of $N^2$ rounds are
needed to send all messages.

There are other architectures that can be implemented cheaply in hardware
and can accomplish this task in a smaller number of rounds. One such
architecture is the shared directional multi-channel.  The combinatorial
abstraction is that we imagine the communication graph as a complete
bipartite graph $K_{N, N}$ (with $N$ vertices on the left, representing
the transmitters of the stations, and $N$ vertices on the right,
representing the receivers).  A directed multi-channel allows us to
partition $K_{N, N}$ into $C$ graphs $G_1, G_2, ... G_C$.  These graphs
correspond to allocating $c$ receivers to each station.  For each graph
$G_i$, in each round we can exchange  all messages corresponding to
the edges in some {\em induced} matching in $G_i$ in one time step.
These matchings are required to be induced because otherwise messages
would interfere in the underlying hardware.

Thus the problem of designing a communication protocol for this 
architecture that completes in a small number of rounds and does not use 
too many transmitters and receivers per station is  
exactly 
the problem of covering all the edges of a complete bipartite 
graph (using at most $C$ graphs) so that the number of 
induced matchings needed to cover the edges in each graph is small.
The number $C$ represents the number of receivers that each 
station must be equipped with, assuming it has only one
transmitter, and so our goal is not only to minimize the 
number of rounds, but also to do so for a small value of $C$.  

\begin{itemize}

\item For $C = 2$, we give a protocol 
that completes in $O(N^{2 - \delta})$ rounds for $\delta > 0.058$ and

\item For any $\eps > 0$, we show that there is a 
$C = C(\eps) = 2^{O(\frac{1}{\eps})}$ so that there 
is a communication protocol that completes in $O(N^{1+\eps})$ rounds

\end{itemize}

Let $K_{N, N}$ be the complete bipartite graph with $N$ vertices 
on the left and $N$  on the right. 

\begin{theorem}
There is a partition of the edges of $K_{N, N}$ into two graphs
$G_1$ and $G_2$ so that each of these graphs can be covered by 
at most $O(N^{2 - \delta})$ induced matchings, for $\delta > 0.058$. 
\end{theorem}

\begin{proof}
This follows immediately from the construction given at the very end of
Section~\ref{sec:ecc}: we can choose $G'_1$ and $G'_2$ that cover
all edges of $K_{N,N}$, where  $G'_1$ covers all edges of $K_{N,N}$ but at
most $N^{2-\delta}$ and yet it is a union of
at most $N^{2 - \delta}$ induced matchings. The second graph $G'_2$ 
consists of all
these remaining edges. Since it contains at most $N^{2 - \delta}$ 
edges in total we can cover $G'_2$ 
by trivial induced matchings -- one for each edge in $G'_2$. 
The total number of induced matchings in each graph is thus
at most $O(N^{2 - \delta})$. 
\end{proof}

\begin{theorem}~\label{thm:close}
For any $\epsilon > 0$, there is a $C = C(\epsilon) = 
2^{O(\frac{1}{\epsilon})}$ so that the edges of $K_{N, N}$ can 
be partitioned into $G_1, G_2, .. G_C$ and each of these graphs 
can be covered by at most $O(N^{1 + \epsilon})$ induced matchings. 
\end{theorem}

\begin{proof}
To obtain this theorem, we can instead invoke the construction 
in Section~\ref{sec:geom} to obtain a bipartite 
graph $G$ (obtained  
by splitting the vertices of the graph constructed in that
section into two equal parts and by keeping all edges that join 
vertices in the two parts). 
For each $i$, we can take $G_i$ to be a random shift 
of $G$ - i.e. we construct $G_i$ by permuting the labels of 
the vertices on the right randomly. 
$G$ is missing less than $N^{2 - \delta}$ edges, 
for $\delta =  2^{-O(\frac{1}{\epsilon})}$, and hence 
if we take $C =  2/\delta$ random shifts the expected number of 
edges that are not covered in any $G_i$ is less than one. 
Hence there is some choice of $G_1, G_2, ... G_C$ that covers the 
edges in the complete bipartite graph and yet the edges in 
each $G_i$ can be covered by at most $O(N^{1 + \epsilon})$  
induced matchings.  We note that the above proof can be derandomized
using the method of conditional expectations, that is, the graphs
$G_i$ can be generated  efficiently  and deterministically.
\end{proof}

Finally we mention a simple 
lower bound for the number of rounds needed, proved by Meshulam
\cite{Me}. This shows that for any
constant number of receivers a super-linear number of rounds is needed:

\begin{prop}[\cite{Me}]
For any partition of the edges of $K_{N, N}$ into $G_1, G_2, ... G_C$, the total number 
of induced matchings needed to cover  
$G_1, G_2, ... G_C$ is at least $b(C) N^{1 + 1 / (2^C -1)}$.
\end{prop}

\begin{proof}
We apply induction on $C$, the result for $C=1$ is trivial. 
Consider the case $C = 2$. Without loss of generality, 
let $G_1$ contain at least half of the edges from the complete 
bipartite graph and suppose that the minimum number of induced 
matchings needed to cover $G_1$ is $N^r$. Then there is an 
induced matching (in this set) that contains at least 
$\frac{1}{2} N^{2 - r}$ edges and hence $G_2$ contains a complete 
bipartite graph where the number of vertices on the left and on the 
right is at least $\frac{1}{4} N^{2 -r}$. Hence the number of induced 
matchings needed to cover $G_2$ is at least $\frac{1}{16} 
N^{4 - 2r}$. Since the quantity $\max(N^r, N^{4 - 2r})$ 
is minimized for $r = 4/3$ the total number of induced 
matchings needed to cover $G_1$ and $G_2$ is at least $\Omega(N^{4/3})$. 

We can iterate the above argument in the general case. Without loss of 
generality let $G_1$ contain at least $\frac{1}{C} N^2$ edges 
and suppose the minimum number of induced matchings needed to 
cover $G_1$ is $N^{r}$. Then the union of $G_2, G_3, ... G_C$ 
contains a complete bipartite graph where the number of vertices 
on the left and on the right is at least $\frac{1}{2C} N^{2 -r}$.
We can assume by induction that the total number of 
induced matchings needed to cover $G_2, G_3, ... G_C$ 
is at least some $b'(C) N^{ (2-r)(1 + 1/(2^{C-1} -1))}$. 
The quantity $$\max(r, (2-r)\times \frac{2^{C-1}}{2^{C-1} -1})$$ 
is minimized for $r = \frac{2^C}{2^C -1}$ and this completes the proof. 
\end{proof}

Hence any protocol requires at least $\Omega(\log \frac{1}{\epsilon})$ receivers per station
to reduce the number of rounds to $O(N^{1 + \epsilon})$. In contrast, the protocol in Theorem~\ref{thm:close}
uses $2^{O(\frac{1}{\epsilon})}$ receivers per station to complete this same task in $O(N^{1 + \epsilon})$ rounds. 

\subsection{Linearity Testing}

Here we observe that our graphs can be plugged in 
the analysis of H{\aa}stad and Wigderson 
\cite{HW} of the graph test of Samorodnitsky and Trevisan \cite{ST} 
to provide a (modest) strengthening. 
We obtain slightly better bounds on the soundness of this test, 
which may be of interest for a particular range of the parameters. 

The classical linearity test of Blum, Luby and Rubinfeld chooses a 
pair of points $x$ and $y$ uniformly at random from the domain 
of a function, and checks if $f(x) + f(y) = f(x + y)$. The test 
accepts $f$ if and only if this condition is met, and indeed this 
test always accepts a linear function and if $f$ is not linear, 
the probability that this test accepts $f$ can be bounded 
by $\frac{1}{2} + \frac{d(f)}{2}$, where $d(f)$ is the maximum 
correlation of $f$ with a linear function \cite{BLR}. 

What if we want to reduce the probability that a function $f$ 
that is not linear passes this test? 
We could perform $r$ independent trials, in which case the probability 
that $f$ is accepted is bounded by $\Big ( \frac{1}{2} + \frac{d(f)}{2} 
\Big )^r$. However such a test queries the function $f$ on $3 r$ 
locations. Motivated by the problem of designing a PCP with optimal 
amortized query complexity and the related problem for linearity
testing,
Samorodnitsky and Trevisan introduced a graph-based linearity test: 
Associate each vertex in an $r$-vertex complete graph with a 
randomly chosen element from the domain of $f$, and for each 
edge check if $f(x) + f(y) = f(x + y)$ where $x$ and $y$ are the 
values associated with the endpoints of the edge. This test 
accepts if and only if all of these conditions are met. 

This test queries the function $f$ on $r + {r \choose 2}$ 
locations and the hope is that the soundness should behave 
approximately like ${r \choose 2}$ independent trials of the 
original linearity test \cite{BLR}. Samorodnitsky and Trevisan \cite{ST}
showed that the soundness of this test is bounded by 
$$\Big ( \frac{1}{2} \Big )^{r \choose 2} + d(f).$$ 

This analysis was subsequently simplified and 
improved by H{\aa}stad and Wigderson 
\cite{HW} - using the known existence 
of graphs that have many edges but can 
be covered by large (disjoint) induced matchings. 
The intuition behind this connection is that an induced matching 
corresponds to independent trials of the original Blum-Luby-Rubinfeld
linearity test (althoug the formal analysis somewhat masks this
intuition). H{\aa}stad and Wigderson \cite{HW} proved:

\begin{theorem}
If $G = (V, E)$ is an $(r, t)$-RS graph, then the graph-test 
for $G$ accepts a function $f$ with probability at most 
$$ e^{-rt/8} + d(f)^{r/4}.$$
\end{theorem}

H{\aa}stad and Wigderson \cite{HW} used the construction of 
Ruzsa and Szemer\'edi \cite{RS} mentioned in the introduction, which
shows that
there are $(r,t)$-RS graphs on $N$ vertices with
$r=\frac{N}{e^{O(\sqrt {\log N})}}$ and $t=N/3$.

We can plug our constructions directly into this theorem to 
obtain slightly better bounds, for some special values of $d(f)$. 
Our constructions are dense, and hence improve the first term in 
the bound, but the second term is slightly worse (although we still 
have $r = N^{1 - o(1)}$). In general, the obtained  bounds 
will be better than either of those in \cite{ST} or \cite{HW} for some
values of $d(f)$. Note that as the complete graph on $N$ vertices
contains every graph on $N$ vertices, these bounds, like the ones of
\cite{ST} and \cite{HW}, provide an upper estimate for the probability 
that the complete graph linearity test on $N$ vertices accepts
a function $f$, showing that it is at most
$$
\mbox{min} ( ~2^{-{N \choose 2}} +d(f), 2^{-N^{2-o(1)}}+d(f)^{N^{1-o(1)}},
2^{-\Omega(N^2)}+d(f)^{N^{1-o'(1)}} ~).
$$
The first term in the minimum  is the bound of \cite{ST}, the second 
is that of \cite{HW}, and the third (in which the $o'(1)$ term is
a bit worse than the one in the second) follows from our graphs.

\subsection{ The Directed Steiner Tree Problem}

In this short subsection we briefly note the connection between our constructions
and a candidate randomized rounding algorithm for the directed Steiner tree problem
that motivated Vempala \cite{Ve} to ask about the existence of certain
$(r,t)$-RS graphs. 

Giving a poly-logarithmic approximation algorithm for the directed Steiner tree problem
is a famous open problem in approximation algorithms. A special case is
the group Steiner tree problem (in an undirected graph), for which Garg, Konjevod and Ravi
gave an elegant, poly-logarithmic approximation algorithm \cite{GKR}. Charikar et al \cite{CCC}
give an approximation algorithm for the directed Steiner tree problem whose approximation guarantee
is $\tilde O(N^{\epsilon})$ for any $\epsilon > 0$, and this guarantee can be made poly-logarithmic
at the cost of running in quasi-polynomial time. 

Even our understanding of the naive linear programming relaxation is quite weak. Zosin and Khuller \cite{ZK}
give a $\Omega(\sqrt{k})$ integrality gap (where $k$ is the number of terminals), but this construction has exponentially many
(in $k$) vertices. Hence we could still hope that the naive relaxation has at most a poly-logarithmic (in $N$) integrality gap. 

Rajaraman and Vempala considered a stronger
relaxation and a candidate rounding algorithm.
In the case in which the support of
the solution to the linear program is a tree, they proved that 
their rounding algorithm achieves a poly-logarithmic
approximation ratio and this analysis is reminiscent of 
the rounding procedure for the group Steiner tree problem \cite{GKR}.

However, even when flow merges in one layer of a layered graph (i.e. when
the fractional solution is not supported on a tree), attempting to analyze
the behavior of the rounding algorithm led Vempala to a combinatorial
conjecture:

\begin{conjecture} \cite{Ve}
Let $G = (U, V, E)$ be an $N \times k$ complete bipartite graph and $N
\geq k$. Let $\cP$ be a partition of the edge set and for a part $p \in
\cP$, let $p_i$ denote the degree of vertex $i$ in $p$ (i.e. the number
of edges of $p$ incident to $i$). Then 
$$
\sum_{i \in U, j \in V} \min
\Big ( 1, \sum_{p \in \cP} \frac{p_i p_j}{|p|} \Big ) \geq C \frac{N
k}{\log N}.
$$

\end{conjecture}

Our constructions yield a negative answer to the above conjecture. In
our negative example we have $N = k$.
To obtain this result, we can instead invoke the construction
in Section~\ref{sec:ecc} to obtain a bipartite graph $H$ (obtained by
duplicating the vertices of the graph constructed in that section). We
can take $\cP$ to be the induced matchings covering $H$ and additionally
we add a part in the partition (consisting of a single edge) for each
edge across the bipartition missing from $H$.

We can upper bound the right hand side as: 
$$
\sum_{i \in U, j \in V}
\min \Big ( 1, \sum_{p \in \cP} \frac{p_i p_j}{|p|} \Big ) \leq \sum_{(i,
j) \in H} \sum_{p \in \cP} \frac{p_i p_j}{|p|} + \sum_{(i, j) \notin H}
1$$ 
$H$ is an $(r,t)$-RS graph (and $r = \Omega(N^{2 - f })$ in our
construction) and so for each part $p$ we have $ \sum_{(i, j) \in H}
\frac{p_i p_j}{|p|} = 1$ because $p$ is an induced matching with respect
to $H$. The number of parts in the partition (ignoring singletons,
which are not in $H$ anyways) is at most $O(N^f)$ and so we can bound
the contribution of the first term by $O(N^f)$. Also, the number of
edges that $H$ is missing (across the bipartition) is at most $O(N^{e})$
and hence we can bound the above sum by $O(N^e + N^f)$ for $e$ and $f$
as in Theorem~\ref{thm:ecccon} (recall that $N = k$). Since we 
can have both
$e$ and $f$ at most $1.942$ it follows that the conjecture is false.

\section{Concluding Remarks and Open Questions}

We have given two constructions of nearly complete graphs that
can be decomposed into large pairwise edge disjoint induced matchings
and described several applications of these graphs.

The main combinatorial open problem that remains is to determine
or estimate more precisely the set of all pairs $(r,t)$ so that
there are $(r,t)$-RS graphs on $N$ vertices. This is interesting
for most values of the parameters, but is of special interest in some
specific range. In particular, if for
$r=\frac{N}{(\log N)^g}$ with $g>1$, one can show that 
$t=o(N)$, this will 
improve the best known upper bound   for the maximum possible 
cardinality of a subset of $\{1,2,\ldots ,N\}$ with no $3$-term
arithmetic progressions-a problem that received a considerable amount
of attention over the years (see \cite{Sa} and its references).

The study of the combinatorial problem above seems to require a variety of
techniques: the known constructions of \cite{RS}, \cite{FNRRS}, \cite{Me}
and the ones given here apply tools from Additive Number Theory,
Coding Theory, low degree representations of Boolean functions
and Geometry, while the proofs of non-existence
rely on the regularity lemma and on combinatorial and entropy based 
techniques.  All of these, however, still leave a wide gap between
the upper and lower bounds for at least some of the range, and it will
be interesting to find additional ideas that will help to 
study this problem.

In all the applications considered here there are still remaining 
open problems. The
communication protocols over a shared directional multi-channel we suggest,
while improving  substantially the existing ones, are still not optimal,
and the problem of deciding the best possible number of rounds for $N$
stations, even with
two receivers per station, is still not settled, although our results
show that it is $N^{2-\delta}$ for some $\delta$ between $0.058$ and
$2/3$. The best possible upper bound for the probability of acceptance
of a function $f$ in
the linearity graph test, using a complete graph of size $N$, is also
not precisely determined 
as a function of $N$ and $d(f)$ (although here the gap between
the upper bounds and the lower bounds is not large-see \cite{HW}.)
Finally, it will be interesting to decide if our graphs can be helpful
in establishing new integrality gap results for the natural  relaxation
of the directed Steiner tree problem, rather than merely estimating
the performance of specific rounding schemes.
\vspace{0.3cm}

\noindent
{\bf Acknowledgment}
We thank Roy Meshulam and Santosh Vempala for helpful comments.

\end{document}